\documentclass{compositio}
\usepackage{amssymb,amsmath,latexsym,lipsum}
\usepackage[inline]{enumitem}
\usepackage{tabulary}
\usepackage{booktabs}
\usepackage{charter}
\usepackage[title]{appendix}
\usepackage{longtable}
\usepackage{amsthm}
\usepackage{euscript}
\usepackage{color}
\usepackage[T1]{fontenc}
\usepackage[utf8]{inputenc}
\usepackage{enumitem}
\usepackage[singlespacing]{setspace}
\usepackage[colorinlistoftodos]{todonotes}
\usepackage{tikz}
\usetikzlibrary {positioning}
\definecolor {processblue}{cmyk}{0.96,0,0,0}
\usepackage{bookmark}
\usepackage{multirow,bigdelim}
\usepackage{mathrsfs}
\usepackage{mathtools}
\usepackage{tikz-cd}
\usepackage{pgfplots}
\usepackage{pgf}
\usetikzlibrary{arrows}

\makeatletter
\newcommand{\inlineitem}[1][]{%
\ifnum\enit@type=\tw@
    {\descriptionlabel{#1}}
  \hspace{\labelsep}%
\else
  \ifnum\enit@type=\z@
       \refstepcounter{\@listctr}\fi
    \quad\@itemlabel\hspace{\labelsep}%
\fi}
\makeatother

\newtheorem{theorem}{Theorem}[section]
\newtheorem{proposition}[theorem]{Proposition}
\newtheorem{lemma}[theorem]{Lemma}
\newtheorem{corollary}[theorem]{Corollary}
\newtheorem{definition}[theorem]{Definition}
\theoremstyle{remark}

\theoremstyle{remark}
\newtheorem{remark}[theorem]{Remark}
\newtheorem{question}[theorem]{Question}

\newtheorem{conjecture}[theorem]{Conjecture}

\renewcommand{\contentsname}{Contents\\{\footnotesize\normalfont}}
\newcommand{\iu}{{i\mkern1mu}}

\DeclareMathOperator{\ord}{ord}

\DeclareMathOperator{\sw}{sw}

\DeclareMathOperator{\Frac}{Frac}

\DeclareMathOperator{\Def}{Def}

\DeclareMathOperator{\dsw}{dsw}

\DeclareMathOperator{\Proj}{Proj}

\DeclareMathOperator{\cond}{cond}

\DeclareMathOperator{\ASW}{ASW}

\DeclareMathOperator{\characteristic}{char}

\pgfplotsset{compat=1.16}

\begin{document}

\title{The refined local lifting problem for cyclic covers of order four}

\author{Huy Dang}
\email{huydang1130@ncts.ntu.edu.tw}
\address{National Center for Theoretical Sciences, Mathematics Division, No. 1, Sec. 4, Roosevelt Rd., Taipei City 106, Taiwan Room 503, Cosmology Building, National Taiwan University}

\thanks{The author is supported by the Vietnam Institute for Advanced Study in Mathematics and the Simons Foundation Grant Targeted for
Institute of Mathematics, Vietnam Academy of Science and Technology.}

\begin{abstract}
Suppose $\phi$ is a $\mathbb{Z}/4$-cover of a curve over an algebraically closed field $k$ of characteristic $2$, and $\Phi_1$ is a \emph{nice} lift of $\phi$'s $\mathbb{Z}/2$-sub-cover to a complete discrete valuation ring $R$ in characteristic zero. We show that there exist a finite extension $R'$ of $R$, which is determined by $\Phi_1$, and a lift $\Phi$ of $\phi$ to $R'$ whose  $\mathbb{Z}/2$-sub-cover isomorphic to $\Phi_1 \otimes_R R'$. That result gives a non-trivial family of cyclic covers where Sa{\"i}di's refined lifting conjecture holds. In addition, the manuscript exhibits some phenomena that may shed some light on the mysterious moduli space of wildly ramified Galois covers.
\end{abstract}

\maketitle

\tableofcontents


\section{Introduction}
As the name suggests, the goal of a \emph{lifting problem} is to construct some objects in characteristic $0$ that ``lift'' the given ones in characteristic $p>0$. It is inspired by a result by Grothendieck, which says that every curve in positive characteristic lifts \cite[III]{MR2017446}. We are interested in the lifting problem for Galois covers of curves, which can be formally stated as below.

\begin{question}[({\emph{The lifting problem for covers of curves}})] Let $k$ be an algebraically closed field of characteristic $p>0$. Suppose we are given a finite group $\Gamma$ and a $\Gamma$-cover $\overline{f}: Y \rightarrow X$ of smooth projective, connected curves over $k$. Is there a $\Gamma$-cover $f: \mathscr{Y} \rightarrow \mathscr{X}$ of smooth curves over a discrete valuation ring $R$ in characteristic $0$ whose special fiber $f \otimes_R k$ is isomorphic to $\overline{f}$?
\end{question}


The answer is NOT always a yes. See \cite[\S 1.1]{MR3051249} for a counter-example. However, Oort speculated (which first appeared in 1995 on a list of questions and conjectures published in \cite[Appendix 1]{MR3971540}) that an arbitrary $\Gamma$-cover should lift when $\Gamma$ is cyclic. It was proved just recently by Obus, Wewers, and Pop in \cite{MR3194815} and \cite{MR3194816}. Moreover, Obus suggested to the author that the techniques used in \cite{MR3194815} may have the potential to confirm a more general form of the conjecture as follows.

\begin{conjecture}{\cite{MR3051252}}({\emph{The refined Oort conjecture}})
Suppose $\phi: Z \xrightarrow{} X$ is a cyclic $\Gamma$-cover of a curve over $k$, and $\phi_1: Y \xrightarrow{} X$ is its Galois sub-cover. Suppose $\Phi_1: \mathcal{Y}_R \xrightarrow{} \mathcal{X}_R$ is a lift of $\phi_1$ to a finite extension $R/W(k)$, hence in charactersitic $0$. Then there exist a finite extension $R'/R$ and a lift $\Phi: \mathcal{Z}_{R'} \xrightarrow{} \mathcal{X}_{R'}$ of $\phi$ over $R'$ that contains $\Phi_1 \otimes_R R': \mathcal{Y}_{R'} \xrightarrow{} \mathcal{X}_{R'}$ as a sub-cover.
\end{conjecture}

Thanks to a local-to-global principle (see \cite[Proposition 2.22]{2020arXiv201013614D}, \cite[\S 3]{MR1424559}), it suffices to show that the following local version holds.

\begin{conjecture}
\label{Conjecturerefinedlocallifting}
Let $k[[z_2]]/k[[x]]$ be a cyclic $\Gamma$-extension. Suppose we are given a discrete valuation ring $R$ in characteristic zero and a lift $R[[Z_2]]/R[[X]]$ of a Galois sub-extension $k[[z_1]]/k[[x]]$. Then there exist a finite extension $R'$ of $R$ and a $\Gamma$-extension $R'[[Z_2]]/ \allowbreak R'[[X]]$ that lifts $k[[z_2]]/k[[x]]$ and contains $R'[[Z_2]]/R'[[X]]$ as a sub-extension.
\end{conjecture}  

That is, one can always fill in the commutative diagram in Figure \ref{fig:diagramconjecture}
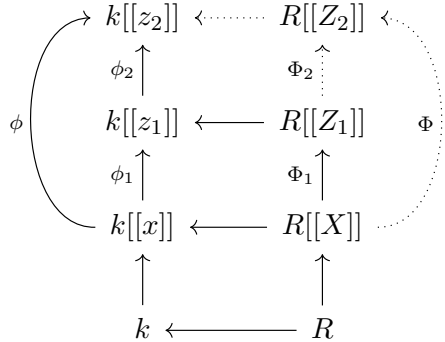
\begin{figure}[h]
    \centering
    \begin{equation*}
\begin{tikzcd}
k[[z_2]]  
& R[[Z_2]]    \arrow[l, dotted] \\
k[[z_1]] \arrow{u}{\phi_2} 
& R[[Z_1]] \arrow[u, dotted, "\Phi_2" left] \arrow[l]  \\
k[[x]] \arrow[u, "\phi_1"] \arrow[uu, bend left=90, "\phi"]    & R[[X]] \arrow[l] \arrow[u, "\Phi_1"] \arrow[uu, dotted, bend right=90, "\Phi"] \\
 k \arrow[u]   &  R \arrow[l] \arrow[u]
\end{tikzcd}
\end{equation*}
    \caption{The refined local lifting problem}
    \label{fig:diagramconjecture}
\end{figure}
where $R$ is a finite extension of the Witt vector $W(k)$, each column is a tower of cyclic extensions, and the group actions on each row are ``compatible'' in the obvious sense. We call Conjecture \ref{Conjecturerefinedlocallifting} the \emph{refined local lifting problem}. Furthermore, as in the standard lifting problem, we may reduce our study to the case when the cyclic group $\Gamma$ is of order $p^n$.
\begin{proposition}
\label{propreducetozpn}
Suppose $\Gamma=\mathbb{Z}/p^n \times \mathbb{Z}/m$, where $m$ is prime to $p$. Then Conjecture \ref{Conjecturerefinedlocallifting} holds for $\Gamma$ if and only if it holds for $\mathbb{Z}/p^n$.
\end{proposition}

\begin{proof}
One can simply generalize the argument of \cite[Proposition 6.3]{MR3051249}. See also \cite[Proposition 2.30]{2020arXiv201013614D}, which proves an analog of Proposition \ref{propreducetozpn} in the equal characteristics case.
\end{proof}

In this paper, we prove the conjecture for the most simple tower in characteristic $2$.

\begin{theorem}
\label{theoremlifttowerz4}
Suppose $\Gamma=\mathbb{Z}/4 \times \mathbb{Z}/m$, where $m$ is odd. Then conjecture \ref{Conjecturerefinedlocallifting} holds in characteristic $2$ when $\Phi_1$ is nice, and $R'$ is determined by $\Phi_1$.
\end{theorem}
Due to Proposition \ref{propreducetozpn}, it suffices to consider the case $m=1$. In particular, suppose we are given a $\mathbb{Z}/4$-cover $\phi$ and a lift $\Phi_1$ of its $\mathbb{Z}/2$-sub-cover $\phi_1$ as in Figure \ref{fig:diagramconjecture}. Then Theorem \ref{theoremlifttowerz4} holds if one can always construct a $\mathbb{Z}/4$-cover $\Phi$ of $R'[[X]]$ that lifts $\phi$ and whose $\mathbb{Z}/2$-subcover is isomorphic to $\Phi_1 \otimes_R R'$ as the diagram suggests. 

\begin{remark}
    The term ``nice lifts'' for $\mathbb{Z}/2$-covers is defined in Definition \ref{defniceZ2}. It is important to note that the ``nice'' assumption makes our result weaker than that of Obus and Wewers \cite[Theorem 3.4 (i)]{MR3194815} as discussed in Remark \ref{remarkconnectionOW}. However, our proof significantly differs from theirs. 
    
    We plan to replace this manuscript with a joint work with Obus that utilizes the Hurwitz tree technique and removes the ``nice'' assumption.
\end{remark}

\begin{remark}
    The main obstacle that prevents us from extending Theorem \ref{theoremlifttowerz4} is generalizing \S \ref{secgeneralforms}, which shows that a lift of a $\mathbb{Z}/2$-cover can be explicitly described. That phenomenon is actually unique to characteristic $2$ as discussed in Remark \ref{remarkOSScomponent}. Furthermore, the construction in \S \ref{secthelift} is quite explicit. Our ability to do so is thanks to the small characteristic of the field, which enables us to systematically address all potential issues.
\end{remark}

\subsection{Acknowledgements}
The author thanks Andrew Obus for helpful discussions and carefully proofreading an earlier version of this paper. He thanks Nathan Kaplan, Joe Kramer-Miller, and the University of California, Irvine's Department of Mathematics for their hospitality when he got stranded in Orange County amid the Covid-19 pandemic. The author
is grateful to the Institute of Mathematics of the Vietnam Academy of Science and Technology and the Vietnam Institute for Advanced Study in Mathematics for the excellent working conditions they provide. This work is funded by the Vietnam Institute for Advanced Study in Mathematics and the Simons Foundation Grant Targeted for the Vietnam Academy of Science and Technology.

\section{Geometric set up}
\label{secsetup}

Recall that $\phi$ is a $\mathbb{Z}/4$-extension $k[[z_2]]/k[[x]]$, where $k$ is an algebraically closed field of characteristic $2$. One may associate with $\phi$ a unique \emph{Katz-Gabber cover}, a.k.a., \emph{HKG cover} (see \cite{MR579791} and \cite{MR867916}) $W_2 \xrightarrow[]{\phi} \mathbb{P}^1_k \cong \Proj k[x,w]$ that is
\begin{itemize}
    \item totally ramified above $x=0$, and such that
    \item the formal completion above $0$ yields $k[[z_2]]/k[[x]]$.
\end{itemize}
From now on, we will identify $\phi$ with its corresponding HKG-over. Theorem \ref{theoremlifttowerz4} is thus equivalent to the following version for one-point-covers. It is not only compatible with the language in \cite{MR3194815} but also allows us to deal with rational functions instead of Laurent series.

\begin{proposition}
\label{propmainonepointcover}
Suppose $k[[z_2]]/k[[x]]$ is a $\mathbb{Z}/4$-Galois extension, and $\Phi_{1}: Y_{1}\xrightarrow{} C:=\mathbb{P}^1_K$ is a $\mathbb{Z}/2$-cover with the following properties.
\begin{enumerate}[label=(\arabic*)]
    \item \label{propmainonepointcover1} The cover $\Phi_1$ has good reduction with respect to the standard model $\mathbb{P}^1_R$ of $C$ and reduces to a $\mathbb{Z}/2$-cover $\phi_{1}: \overline{Y}_1 \xrightarrow{} \overline{C} \cong \mathbb{P}^1_k$ that is totally ramified above $x=0$ and {\'e}tale elsewhere.
    \item \label{propmainonepointcover2} The completion of $\phi_1$ at $x=0$ yields $k[[z_1]]/k[[x]]$, the unique $\mathbb{Z}/2$-sub-extension of the given extension $k[[z_2]]/ \allowbreak k[[x]]$.
\end{enumerate}
Then $\Phi_1$ extends to a $\mathbb{Z}/4$-cover $\Phi_2: Y_2 \xrightarrow{} C$ with good reduction such that
\begin{enumerate}[label=(\alph*)]
    \item \label{propmainonepointcover3} Its reduction $\phi_2: \overline{Y}_2 \xrightarrow{} \overline{C}$ is totally ramified above $\overline{x}=0$, {\'e}tale everywhere else, and
    \item \label{propmainonepointcover4} The completion of $\phi_2$ at $x=0$ yields $k[[z_2]]/k[[x]]$.
\end{enumerate}
\end{proposition}

\subsection{Kummer-Artin-Schreier-Witt theory}
\label{secKASW}
Let $\mathbb{K}$ be the function field of the projective line over a field $K$. We set
\begin{equation}
    \label{eqncharacter}
    {\rm H}^1_{p^n}(\mathbb{K}):={\rm H}^1(\mathbb{K}, \mathbb{Z}/p^n).
\end{equation}
We call an element of ${\rm H}^1_{p^n}(\mathbb{K})$ a \emph{character}. As in \cite{MR3194815}, we will use such a character as a substitute for an exponent-$p^n$-extension of $\mathbb{K}$ since they are more convenient to manipulate algebraically. For instance, if $\chi_n \in {\rm H}^1_{p^n}(\mathbb{K})$, then one may regard $(\chi_n)^{p^{n-i}}$, for $1 \le i < n$, as $\chi_n$'s $\mathbb{Z}/p^{i}$-sub-cover. In addition, if $\mathbb{K}$ has characteristic $0$ (resp., characteristic $p$), then 
\begin{equation}
    \label{eqnkummerASW}
    {\rm H}^1_{p^n}(\mathbb{K}) \cong \mathbb{K}/\mathbb{K}^{p^n} \hspace{5mm} (\text{resp., }  {\rm H}^1_{p^n}(\mathbb{K}) \cong W_n(\mathbb{K})/\wp(W_n(\mathbb{K})),
\end{equation}
by Kummer theory (resp., Artin-Schreier-Witt (ASW) theory) given that $\mathbb{K}$ contains a primitive $p^n$-th root of unity.

A character $\chi \in {\rm H}^1_{p^n} (\mathbb{K})$ gives rise to a (possibly) branched Galois cover $Y \xrightarrow{f} \mathbb{P}^1_K$. A point $x \in \mathbb{P}^1_K$ is called a \emph{branch point} of $\chi$ if it is a branch point of $f$. The \emph{branching index} of $x$ is the order of the inertia group of a point on $x$'s fiber. 

\begin{definition}
With the notations above and $\characteristic(\mathbb{K})=0$, an element $F \in \mathbb{K}$ gives rise to a character $\mathfrak{K}_{p^n}(F) \in \text{\rm H}^1_{p^n}(\mathbb{K})$ defined by the Kummer equation $W_n^{p^n}=F$. In particular, if 
$$ F=\prod_{i=0}^{n-1} F_i^{p^i} \not\in \mathbb{K}^{p^n} $$
where $F_i \in \mathbb{K} \setminus \mathbb{K}^p$ or $F_i=1$ for each $i$, then the $\mathbb{Z}/p^j$-subcover ($j<n$) is precisely
$$ W_j^{p^j}= \prod_{i=0}^{j-1} F_i^{p^i} .  $$
Suppose $\characteristic (\mathbb{K})=p>0$. Then a length-$n$-Witt-vector $\underline{f}=(f_1, \ldots, f_n) \in W_n(\mathbb{K})$ gives rise to a character $\mathfrak{K}_{p^n}(\underline{f}) \in H^1_{p^n}(\mathbb{K})$ corresponding to the ASW extension
$$ \wp(y_1, \ldots, y_n)= (f_1, \ldots, f_n), $$
where $\wp$ is the ASW isogeny. See \cite[\S 26]{MR2371763} for more details. We write $\underline{f} \sim_{\ASW} \underline{g}$ if they are in the same ASW class, i.e., $\underline{g}=\underline{f}+ \wp (\underline{h})$ for some $\underline{h} \in W_n(\mathbb{K})$. We call $\chi \in H^1_{p^n}(\mathbb{K})$ an \emph{admissible character} if its branch locus, usually denoted by $\mathbb{B}(\chi)$, specializes to a single point.
\end{definition}

\subsection{Branching data of a cyclic cover in characteristic \texorpdfstring{$p$}{p}}
\label{secbranchingdata}
Recall from the previous section that we may identify a $\mathbb{Z}/p^n$-cover $\chi_n: Y_n \rightarrow \mathbb{P}^1_k$ with a (class of) length-$n$-Witt vector $$\underline{g}=(g_1, \ldots, g_2) \in W_n(k(x))/\wp(W_n(k(x))).$$  
We say $\underline{g}$ is \emph{reduced} if its entries' partial fraction decompositions have no terms that are $p$th powers. The perfectness assumption on $k$ then implies that every class of $W_n(k(x))/\wp(W_n(k(x)))$ has a unique reduced representation.

Furthermore, at each ramified point $Q_j$ above $P_j \in \mathbb{B}(\chi_n)$, $\chi_n$ induces an exponent-$p^n$-extension of complete local ring $\hat{\mathcal{O}}_{Y_n,Q_j}/\hat{\mathcal{O}}_{\mathbb{P}^1,P_j}$. Hence, it makes sense to talk about the ramification filtration of $\chi_n$ at $P_j$. Suppose the inertia group of $Q_j$ is $\mathbb{Z}/p^m$ (where $n\le m$). We say the \textit{$i$-th ramification break} of $\chi_n$ at $P_j$, denoted by $m_{j,i}$, is $-1$ for $i \le n-m$. When $i >n-m$, the $i$-th ramification break of $\chi_n$ at $P_j$ is the $(i-n+m)$-th one of $\hat{\mathcal{O}}_{Y_n,Q_j}/\hat{\mathcal{O}}_{\mathbb{P}^1,P_j}$. The following result indicates that these invariants can be deduced easily from the reduced representation.

\begin{theorem}[{\cite[Theorem 1]{MR1935414}}]
\label{theoremcaljumpirred}
With the notations above and $\underline{g}$ is reduced, we have
\begin{equation}
\label{eqnformulalowerjumpasw}
    m_{j,i}=\max\{ p^{i-l} \deg_{(x-P_j)^{-1}} (g_{l}) \mid l=1, \ldots, i\}, 
\end{equation}
for $i>n-m$.
\end{theorem}

\subsubsection{}
Let us get back to the situation of Proposition \ref{propmainonepointcover}. We associate with $\phi$ a reduced length-two-Witt vector
\begin{equation*}
    \underline{f}=(f_1, f_2) \in W_n(k(x)).
\end{equation*}
It follows from Theorem \ref{theoremcaljumpirred} that the $f_i$'s only have a pole at $x=0$. Let $(m_1, m_2) \in \mathbb{Z}^2$ be the ramification breaks at that place. The following is immediate from the same result.

\begin{corollary}
With the above notations, the following holds.
\begin{enumerate}[label={\arabic*.}]
    \item $m_1$ is odd, and
    \item $m_2 \ge 2m_1$, if $m_2>2m_1$ then $m_2$ is odd.
\end{enumerate}
\end{corollary}

Furthermore, \cite[Lemma 2.1.2]{MR2016596} implies that the $\mathbb{Z}/2$-extension of the ring of power series $k[[z_1]]/k[[x]]$, is completely determined by $m_1$. Therefore, one may assume that $\phi_1$ is an Artin-Schreier cover of $\mathbb{P}^1_k$ represented by
$$ y_1^2-y_1=\frac{1}{x^{m_1}}. $$
Hence, the Witt vector associated with $\phi$ is of the form
\begin{equation}
\label{eqnphi}
    \bigg(\frac{1}{x^{m_1}},  \frac{\sum_{0 \le i <n_2} a_i x^i}{x^{n_2}} \bigg) \in W_2(k(x)),
\end{equation}
where $n_2$ is odd, $a_i \in k$,  $a_0 \neq 0$, and $a_i=0$ for $i$ odd. In addition, if $n_2>2m_1$ then $n_2=m_2$, if $n_2<2m_1$ then $m_2=2m_1$. 

\section{Degeneration of cyclic covers}
\label{secdegenerationcyclic}

\subsection{Degeneration type}
\label{secdegenerationtype}
Let us briefly introduce the notion of degeneration type, a useful invariant for an abelian local covering like ones in this paper. For more details, see, e.g., \cite[\S 5.3.1]{MR3194815} for the mixed characteristics case or \cite[\S 3.5]{2020arXiv201013614D} for the equal characteristics case.

\begin{definition}
\label{defndegenerationtype}
Suppose $\mathscr{K}$ is a valued field with valuation $\nu$, uniformizer $\pi$, and residue field $\kappa$ of characteristic $p>0$. Let $\chi \in {\rm H}^1_{p^n}(\mathscr{K})$ be a character or order $p^n$. The degeneration of $\chi$ can be measured by its associated \emph{refined Swan conductors}. Those include the followings:
\begin{enumerate}[label={\arabic*.}]
    \item The \emph{depth Swan conductor} $\sw({\chi}) \in \mathbb{Q}_{\ge 0}$, which measures the separability of $\chi$'s reduction. Particularly, it is $0$ if and only if $\chi$ is unramified with respect to $\nu$,
    \item For $\sw(\chi)>0$, the \emph{differential Swan conductor} $\dsw(\chi) \in \Omega^1_{\kappa}$, and
    \item For $\sw(\chi)=0$, a \emph{reduction} $\underline{f}=(f_1, \ldots, f_n) \in W_n(\kappa)$, which can be replaced by another vector in the same ASW class.
\end{enumerate}
We call the pair $(\sw(\chi), \dsw(\chi))$ when $\sw(\chi)>0$ (resp., $(0, \underline{f})$ when $\sw(\chi)=0$) the (resp., a) \emph{degeneration type} (resp., \emph{reduction type}) of $\chi$. 
\end{definition}

These information can be calculated easily when the character is \emph{admissible} of order $p$ and $K$ is of characteristic $0$. The equi-characteristic case is similar \cite[Proposition 3.17]{2020arXiv200203719D}.

\begin{theorem}[{\cite[Proposition 5.17]{MR3194815}}]
\label{theoremswanorderpmixedchar}
Suppose $\mathscr{K}$ is $p$-adic with valuation $\nu$, residue $\kappa$, $\nu(p)=1$, $F \in \mathscr{K}^{\times} \setminus (\mathscr{K}^{\times} )^p$ and $\nu(F)=0$. Suppose, moreover, that $\chi:=\mathfrak{K}_p(F) \in {\rm H}^1_p(\mathscr{K})$ is admissible and weakly unramified with respect to $\nu$ (see \cite{MR0321929}).
\begin{enumerate}[label=({{\alph*}})]
    \item \label{theoremswanorderpmixedchar1} We have
    $$ \sw(\chi)=\frac{p}{p-1}- \min \bigg\{ \max_H \{ \nu(F-H^p)\}, \frac{p}{p-1} \bigg\}, $$
    where $H$ ranges over all elements of $\mathscr{K}$.
    \item \label{theoremswanorderpmixedchar2} The maximum of $\nu(F-H^p)$ in \ref{theoremswanorderpmixedchar1} is achieved if and only if $$g:=[F-H^p] \not\in \kappa^p, $$
    where $[F-H^p]$ is the image in $\kappa$ of $(F-H^p)/p^{\nu(F-H^p)}$. We call $H$ a \emph{correcting function} for $F$.
    If this is the case, and $\sw(\chi)>0$, then
    \[\dsw(\chi)= \begin{cases} 
      dg/g & \text{if } \sw(\chi)=p/(p-1) \\
      dg & \text{if } 0<\sw(\chi) <p/(p-1)
   \end{cases}.
\]
       If, instead, $\sw(\chi)=0$, then its reduction $\overline{\chi}$ corresponds to the Artin-Schreier extension given by the equation $y^p-y=g$. 
\end{enumerate}
\end{theorem}

\begin{definition}
\label{defnbestform}
Suppose $\mathscr{K}$ is $p$-adic with valuation $\nu$ and residue $\kappa$. Set $\lambda:=\zeta_p-1$, which has valuation $1/(p-1)$. Suppose $F \in \mathscr{K}$, and  $\psi=\mathfrak{K}_p(F) \in {\rm H}_p^1(\mathscr{K})$ has Swan conductor $\sw(\psi)=:\delta$. We say $F$ is a \emph{best representation} for $\psi$ if (after maybe a finite extension of $\mathscr{K}$)
\begin{equation*}
    F=1+ \lambda^p\frac{G}{\pi_{\delta}},
\end{equation*}
for $G \in \mathscr{K}$, $\nu(\pi_{\delta})=\delta$, $\nu (G)=0$,  and its image $g$ in $\kappa$ is not a $p$-power.
\end{definition}

\begin{remark}
    With the notations of Theorem \ref{theoremswanorderpmixedchar} \ref{theoremswanorderpmixedchar2}, the fraction $F/H^p$ is a best representation for $\chi$.
\end{remark}

The following result demonstrates how one can compute the degeneration type of a product of characters, which will later allow us to construct $\Phi$ by adjusting a ``minimal covering''.

\begin{lemma}[{\cite[Proposition 5.9]{MR3194815}}]
\label{lemmacombination}
Let $\chi_i \in {\rm H}^1_{p^n}(\mathscr{K})$, with $i=1,2,3$ satisfying the relation $\chi_3=\chi_1 \cdot \chi_2$. Set $\delta_i:=\sw(\chi_i)$ and $\omega_i:=\dsw(\chi_i)$, for $i=1,2,3$. Then the following holds.
\begin{enumerate}[label=(\arabic*)]
    \item \label{lemmacombination1} If $\delta_1 \neq \delta_2$, then $\delta_3=\max \{\delta_1,\delta_2 \}$. Furthermore, we have $\omega_3=\omega_1$ if $\delta_1>\delta_2$ and $\omega_3=\omega_2$ otherwise.
    \item \label{lemmacombination2} If $\delta_1=\delta_2>0$ and $\omega_1+\omega_2 \neq 0$, then $\delta_1=\delta_2=\delta_3$ and $\omega_3=\omega_1+\omega_2$.
    \item \label{lemmacombination3} If $\delta_1=\delta_2>0$ and $\omega_1+\omega_2=0$, then $\delta_3 < \delta_1$.
    \item \label{lemmacombination4} if $\delta_1=\delta_2=0$, then $\delta_3=0$ and $\overline{\chi}_3=\overline{\chi}_1 \cdot \overline{\chi}_2$. Hence, if a reduction type of $\chi_1$ (resp., $\chi_2$) is $\underline{f}$ (resp., $\underline{g}$) in $W_n(\kappa)$, then $\underline{f}+\underline{g}$, which could be $\underline{0}$, is a reduction type of $\chi_3$.
\end{enumerate}
\end{lemma}

\subsection{A criterion for good reduction}
\label{seccriteriongoodreduction}
In this subsection, we consider Definition \ref{defndegenerationtype}'s situation with $\kappa=l(x)$, where $l$ is a perfect field of characteristic $p$. As discussed in \ref{secbranchingdata}, when $\sw(\chi)=0$, we call the reduced vector in the same ASW class of $\chi$'s reduction its \emph{reduced reduction}.

\begin{definition}
\label{defnconductor}
With the notations of Definition \ref{defndegenerationtype} and $\kappa=l(x)$, we may view $\chi$ as a $\mathbb{Z}/p^n$-cover of $\mathbb{P}^1_{\mathscr{K}}$. For each branch point $P$ of $\chi$'s generic fiber $\chi_{\eta}$, we define the \emph{conductor} of $\chi_{\eta}$ at $P$, denoted by $\cond(\chi_{\eta},P)$, as follows
\begin{equation*}
    \cond(\chi_{\eta},P)=\begin{cases} 
      1, & \characteristic(\mathscr{K})=0  \\
      \text{($\chi_{\eta}$'s highest ramification break at $P$}) +1, & \characteristic(\mathscr{K})=p \\
   \end{cases}.
\end{equation*}
When $\sw(\chi)=0$, for each pole $q$ of the reduced reduction $\underline{f}$ we define the conductor at $q$ to be the conductor at the corresponding branch point of $\mathfrak{K}_{p^n}(\underline{f}) \in {\rm H}^1_{p^n}(\kappa)$.
\end{definition}

Recall that the special fiber's highest ramification breaks can be calculated from $\underline{f}$ using Theorem \ref{theoremcaljumpirred}. One then can characterize the reduction of $\chi$ from its degeneration type as follows.

\begin{proposition}
\label{propvanishingcycle}
Suppose $\chi$ is a character in ${\rm H}^1_{p^n}(\mathscr{K})$ with residue $\kappa=l(x)$. Suppose moreover that $\sw(\chi)=0$ with $\underline{f} \in W_n(\kappa)$ the reduced reduction. Then for each pole $x_i$ of $\underline{f}$ and let $\{X^j_i\}_{j \in J}$ be the branch points of $\chi$'s generic fiber $\chi_{\eta}$ specializing to $x_i$, we have
    \begin{equation}
        \label{eqndiffvanishingcycle}
        \cond(\underline{f},x_i) \le \sum_{j \in J} \cond(\chi_{\eta}, X^j_i).
    \end{equation}
Moreover, $\chi$ has good reduction if and only if the equality holds.
\end{proposition}

\begin{proof}
Notice that for each $i$ the completion of the localization $\chi^i$ of $\chi$ at $x_i$ gives an admissible character which was studied in \cite{MR3194815} and \cite{2020arXiv201013614D}. When $\sw(\chi)=0$, we have $$\sw_{\mathfrak{K}_{p^n}(\underline{f})}(x_i) = \sw_{\chi^i}(x_i)=\cond(\underline{f}, x_i)-1. $$ The first equality follows from \cite[Remark 5.8 (i)]{MR3194815}, where $\sw_{\chi^i}(x_i)$ is the \emph{boundary Swan conductor} of $\chi^i$  at $x_i$, and $\sw_{\mathfrak{K}_{p^n}(\underline{f})}(x_i)$ is the usual Swan conductor of $\mathfrak{K}_{p^n}(\underline{f})$ with respect to $\ord_{x_i}$, which is in turn equal to the highest ramification break of $\mathfrak{K}_{p^n}(\underline{f})$ at $x_i$. Furthermore, the right-hand-side of (\ref{eqndiffvanishingcycle}) is precisely $\lvert \mathbb{B}(\chi^i) \rvert$ in \cite[Corollary 5.13]{MR3194815} (characteristic zero) or  $\mathfrak{C}_{\chi^i}(0, 0, x_i)$ in \cite[Corollary 3.24]{2020arXiv201013614D} (characteristic $p$). The proposition follows at once. All the cited results are consequences of \cite[Theorem 6.7]{MR904945}.
\end{proof}

\section{The universal \texorpdfstring{$\mathbb{Z}/2$}{}-deformation}
\label{secgeneralforms}
The below result shows that a $\mathbb{Z}/2$-lift can be explicitly described. 
\begin{proposition}
\label{proprepresentationoforder2}
Suppose we are given an Artin-Schrier extension of $\kappa=k(x)$ defined by
\[ y^2-y=\frac{1}{x^{m_1}}:=f_1, \]
where $m_1=2q_1+1$. Suppose $\Phi_1$ is a lift of $\phi_1$ to a complete discrete valuation ring $R$ in characteristic zero with valuation $\nu$ normalized so that $\nu(2)=1$, uniformizer $\pi$. Then $\Phi_1$ can be represented by a Kummer equation of the form
\begin{equation}
    \label{eqnZ/2liftreducedform}
    W_1^2=1+4 \frac{X \cdot G}{\prod_{i=1}^{q_1+1}(X-v_i)^2},
\end{equation}
where the $v_i$'s lie in a finite extension $R_{\Phi_1}/R$, and $\nu(v_i)>0$, $G \in R[X]$, and $G \equiv 1$ modulo $\pi$. In particular, we have
\begin{equation}
\label{eqnZ/2liftformnumerator}
    \prod_{i=1}^{q_1+1}(X-v_i)^2+4 X \cdot G=\prod_{j=1}^{m_1+1}(X-B_i),
\end{equation} where $B_1, \ldots, B_{m_1+1} \in R$ are distinct branch points of $\Phi_1$.
\end{proposition}

\begin{proof}
Set $\mathbb{K}:=\Frac (R[X])$ and equip it with the Gauss valuation $\nu_{\mathbb{K}}$ so that $\nu_{\mathbb{K}}(X)=0$ and $\nu_{\mathbb{K}}(2)=1$. Consider $\Phi_1$ as a character in ${\rm H}^1_2(\mathbb{K})$. Let $\mathbb{B}(\Phi_1)$ be $\Phi_1$'s generic branch locus. As $\Phi_1$ has good reduction, one obtains the following equation when apply Proposition \ref{propvanishingcycle}
\begin{equation*}
\begin{split}
    \cond (f_1, 0) &=\sum_{B \in \mathbb{B}(\Phi_1)} \cond(\Phi_{1,\eta},B) \\
    m_1+1 &= \# \mathbb{B}(\Phi_1).
\end{split}
\end{equation*}
We thus may assume that $\Phi_1$'s branch locus has $m_1+1$ distinct points $B_1, B_2 \ldots, B_{m_1+1}$, all specializing to $0$. Kummer theory for order $2$ then indicates that $\Phi_1$ is defined by
\[ W_1^2= \prod_{j=1}^{m_1+1}\bigg(1-\frac{B_i}{X}\bigg) =: 1 + \sum_{i=1}^{m_1+1} \frac{b_i}{X^i}=:F \in R[X^{-1}]. \]
Note that the order of $F$'s denominator at $X=0$ is $m_1+1$, which is even, hence does not contribute to $\mathbb{B}(\Phi_1)$. One also sees that $\nu_{\mathbb{K}}(F)=0$. Applying \cite[Lemma 5.18]{MR3194815} with $N=q_1+1$, one finds $H=1+\sum_{j=1}^{q_1+1}\frac{d_j}{X^j} \in R[X^{-1}]$ such that $$F-\text{\rm H}^2=\sum_{l=1}^{m_1} \frac{a_l}{X^l} \in R[X^{-1}] $$ verifying $a_{2l}=0$ for all $l \le q_1$. Thus, we have $\nu_{\mathbb{K}}(F-\text{\rm H}^2) \ge 0$. In addition, the image of $(F-\text{\rm H}^2)2^{-\nu_{\mathbb{K}}(F-\text{\rm H}^2)}$ in $\kappa$ (via the canonical reduction map), which we denote by $[F-\text{\rm H}^2]$, is not a square in $\kappa$ since all the nonzero coefficients are of odd degree. Therefore, as $\Phi_1$ has a good reduction isomorphic to $\phi_1$ by the hypothesis, Theorem \ref{theoremswanorderpmixedchar} asserts that $\nu_{\mathbb{K}}(F-\text{\rm H}^2)=2$ and $[F-H]=1/x^{m_1}$. We hence obtain a best representation for $\Phi_1$
\[ \frac{F}{\text{\rm H}^2}=1+\frac{F-H^2}{H^2} = 1+4\frac{X \cdot G}{(X^{q_1+1}+X^{q_1}d_1+ \ldots +d_{q_1+1})^2}=:1+4 \frac{X \cdot G}{(H')^2} , \] where $G \in R[X]$, and $G \equiv 1 \pmod{\pi}$. Let $R_{\Phi_1}$ be the minimal extension of $R$ such that $H'$ splits completely into $\prod_{i=1}^{q_1+1} (X-v_i)$. Set $\tilde{\mathbb{K}}=\Frac R_{\Phi_1}[X]$. Finally, as $\nu_{\mathbb{K}}(d_j)\ge jr_0$ for some $r_0>0$ as shown in the proof of \cite[Lemma 5.18]{MR3194815}, the Newton polygon technique (see \cite[Proposition 6.3]{MR1697859}) affirms that $\nu_{\tilde{\mathbb{K}}}(v_i)>0$ for $i=1, \ldots, q_1+1$, hence the final assertion.
\end{proof}

\subsection{The versal deformation ring of an Artin-Schreier \texorpdfstring{$\mathbb{Z}/2$}{Z2}-cover}

Let $\hat{\mathcal{C}}$ denote the category of local Noetherian complete $K:=\Frac(W(k))$-algebras $A$ with residue $k \cong A/\mathfrak{m}_A$. Let $\mathcal{C}$ be $\hat{\mathcal{C}}$'s subcategory consisting of $K$-modules of finite type. Let $\Def^{m_1}_{\mathbb{Z}/2}: \mathcal{C} \rightarrow \text{SETS}$ be the deformation functor defined by
    \begin{equation*}
        \Def^{m_1}_{\mathbb{Z}/2}(A)=\big\{ \text{(smooth) deformations of $\phi_1: k[[x]] \rightarrow k[[z_1]]$ to $A$} \big\}/\cong.
    \end{equation*}
\cite[Th{\'e}or{\`e}m 2.2]{MR1767273} shows that the functor admits a versal deformation. In the following, we give an alternative proof for \cite[Th{\'e}or{\`e}m 4.3.7]{MR1767273}'s characteristic $2$ case using the tools in \S \ref{secdegenerationcyclic}.

\begin{proposition}
\label{propASchar2versaldeformation}
    Suppose we are given a $\mathbb{Z}/2$-extension $k[[z_1]]/k[[x]]$ with conductor $m_1+1=2q_1+2$.
    Suppose $\Tilde{\Phi}_1$ is a deformation of $\phi_1$ over $R_{m_1}:=K[[u_1, \ldots, u_{q_1+1}]]$ given by
    \begin{equation}
    \label{eqnuniversalZ/2deformation}
        W_1^2=1+4\frac{X}{(X^{q_1+1}+u_1 X^{q_1}+ \ldots + u_{q_1+1})^2}.
    \end{equation}
    Then the universal pair $(R_{m_1}, \Tilde{\Phi}_1)$ represents $\Def^{m_1}_{\mathbb{Z}/2}$. 
\end{proposition}

\begin{proof}[Sketch of proof]
The claim that $\Tilde{\Phi}_1$ smoothly deforms $\phi_1$ is proved in \cite[Lemma 4.3.1]{MR1767273}. We paraphrase it here for the convenience of the reader. One first observes that substituting $W_1=1+2Y_1$ to (\ref{eqnuniversalZ/2deformation}) yields the following Artin-Schreier like equation
\begin{equation}
\label{eqnuniversalZ/2AStype}
    (Y_1)^2-Y_1=\frac{X}{(X^{q_1+1}+u_1 X^{q_1}+ \ldots + u_{q_1+1})^2}.
\end{equation}
We hence recover the equation represented $\phi_1$ when $u_i=0$. Thus, $\tilde{\Phi}_{1}$'s special fiber is birational to $\phi_1$. Let $B$ be the normalization of $R_{m_1}[[X]]$ in the extension of its fraction field by adjoining $W_1$. Substitute $Z_1:=(1-W_1)(X^{q_1+1}+u_1 X^{q_1}+ \ldots + u_{q_1+1})/2$ to (\ref{eqnuniversalZ/2deformation}), one obtains
\begin{equation}
\label{eqnZ/2universalnormal}
    Z_1^2-Z_1(X^{q_1+1}+u_1 X^{q_1}+ \ldots + u_{q_1+1})-X =0.
\end{equation}
Therefore, $Z_1 \in B$ and (\ref{eqnZ/2universalnormal}) is of the form $F(Z_1,X)=0$ with $\frac{\partial F}{\partial X} =1 \mod{(Z_1, X)}$. Hensel's lemma then tells us that $X$ lies in $R_{m_1}[[Z_1]]$, hence $B=R_{m_1}[[Z_1]]$, hence $\tilde{\Phi}_1$ has good reduction isomorphic to $\phi_1$. 

``$\Rightarrow$'' Suppose $A$ is a complete valued ring with valuation $\nu$ in $\mathcal{C}$. Then an algebra homomorphism $\psi: R_{\phi_1} \rightarrow A$ is determined by the identification $\psi(u_i)= v_i$, where $v_1, \ldots, v_{q_1+1} \in A$ (may not be distinct) with $\nu(v_i)>0$. The specialization of $\Tilde{\Phi}_1$ under $\psi$ gives rise to a $\mathbb{Z}/2$-extension $\Phi_{1,\psi}:=\Tilde{\Phi}_1 \otimes_{R_{\phi_1}} A$ of $A[[X]]$. The following will discuss why $\Phi_{1,\psi}$ is always a deformation of $\phi_1$ over $A$ using our criterion for good reduction in \S \ref{seccriteriongoodreduction}.

When $A$ is of mixed characteristic $(0,p)$, then the fact that $\Phi_{1, \psi}$ birationally deforms $\phi_1$ is also instantaneous from Theorem \ref{theoremswanorderpmixedchar}. In addition, due to Kummer theory, $\Phi_{1,\psi}$'s generic branch locus are the zeros of the degree $2q_1+2=m_1+1$ polynomial $(X^{q_1+1}+v_1 X^{q_1}+ \ldots + v_{q_1+1})^2+4X$, hence has size at most $m_1+1$. The remaining is immediate from Proposition \ref{propvanishingcycle}.

When $A$ is of equal-characteristic $2$, it follows from (\ref{eqnuniversalZ/2AStype}) that $\Phi_{1, \psi}$ is generically Artin-Schreier of the form 
\begin{equation*}
    (Y_1)^2-Y_1= \frac{X}{(X^{q_1+1}+v_1 X^{q_1}+ \ldots + v_{q_1+1})^2}=: \frac{X}{\prod_{j=1}^l(X-w_j)^{2n_j} },
\end{equation*}
where the $w_j$'s are distinct and $\sum_{j=1}^l n_j=q_1+1$. Furthermore, the conductor (Definition \ref{defnconductor}) of $\Phi_{1,\psi}$'s generic fiber, which we denote by $\Phi_{1, \psi, \eta}$, at each place $w_j$ is at most $2n_j$ due to Theorem \ref{theoremcaljumpirred}. Applying Proposition \ref{propvanishingcycle} to $\Phi_{1, \psi}$ gives us the following inequality
\begin{equation*}
    \cond(\phi_1, 0)=2q_1+2 \le \sum_{i=1}^l \cond(\Phi_{1, \psi, \eta}, w_i) \le \sum_{i=1}^l 2 n_i=2q_1+2.
\end{equation*}
Therefore, the same proposition asserts that $\Phi_{1, \psi}$ is a deformation of $\phi_1$. 

``$\Leftarrow$'' Conversely, suppose we are given $\Phi_1 \in \Def^{m_1}_{\mathbb{Z}/2}(A)$. Let us first consider the case $A$ is of mixed characteristic $(0,p)$. Proposition \ref{proprepresentationoforder2} then permits us to assume that $\Phi_1$ is of the form (\ref{eqnZ/2liftreducedform}). As $G$ is congruent to $1$ modulo $\pi$, Hensel's lemma indicates the existence of  $G$'s $m_1$-th root in $R[[X]]$, which we denote by $\sqrt[m_1]{G}$. The change of variables $X \mapsto X \sqrt[m_1]{G}$ transforms (\ref{eqnZ/2liftreducedform}) to
\begin{equation*}
    W_1^2= 1+ 4 \frac{X}{\prod_{i=1}^{q_1+1}(X-v_i/\sqrt[m_1]{G})^2}=: 1+ 4 \frac{X}{(X^{q_1+1}+w_1 X^{q_1}+\ldots +w_{q_1+1})^2}.
\end{equation*}
Thus, the specialization $\Tilde{\Phi}_{1, \psi}$ with $\psi(u_i)= w_i$ is isomorphic to $\Phi_1$. 

Let us now consider the case $A$ is equicharacteristic $2$, and $\Phi_1$ is a deformation with $\mathbb{B}(\Phi_1)=\{ w_1, \ldots, w_l \}$, where $w_i \in \mathfrak{m}_A$ for all $i$. Set $2n_i:= \cond(\Phi_{1,\eta},w_i)$ with $\Phi_{1, \eta}$ being $\Phi_1$'s generic fiber. Once more, Proposition \ref{propvanishingcycle} asserts that $\sum_{i=1}^l 2n_i=2(q_1+1)$. Moreover, if $f_{\Phi_1}$ is the right-hand-side of the Artin-Schreier equation that gives $\Phi_1$, \cite[Proposition 3.17]{2020arXiv200203719D} indicates the following (after possibly replacing $f_{\Phi_1}$ by one in the same Artin-Schreier class)
$$ f_{\Phi_1}  \equiv \frac{1}{X^{2q_1+1}}  \equiv \frac{X}{\prod_{i=1}^l (X-w_i)^{2n_i}} \pmod{\mathfrak{m}_A}.$$
The Katz-Gabber cover technique (\S \ref{secsetup}) again allows us to assume that $f_{\Phi_1}$ is a rational function in $\Frac(A[X])$, thus of the following form
\begin{equation}
\label{eqnfirstreducedfphi1}
    f_{\Phi_1}=\frac{X}{\prod_{i=1}^l (X-w_i)^{2n_i}} + \bigg( \sum_{i=1}^l \sum_{j = 1}^{l_i} \frac{d_{i,j}}{(X-w_i)^j} \bigg),
\end{equation}
where $d_{i,j} \in \mathfrak{m}_A$. In addition, over a finite extension of $A$, one may simplify the terms $d_{i,2q}/(X-w_i)^{2q}$ inductively starting from the highest degree of $(X-w_i)^{-1}$ by adding 
\begin{equation*}
    \wp\bigg( \frac{\sqrt{d_{i,2q}}}{(X-w_i)^{q}}\bigg)=\frac{d_{i,2q}}{(X-w_i)^{2q}}+\frac{\sqrt{d_{i,2q}}}{(X-w_i)^{q}}
\end{equation*}
to (\ref{eqnfirstreducedfphi1}). Furthermore, the assumption on generic conductors and Theorem \ref{theoremcaljumpirred} indicate that there exists a rational function $b$ in (a finite extension of) $\Frac (A[X])$ with no poles outside $\mathbb{B}(\Phi_1)$ such that $\ord_{(X-w_i)}( f_{\Phi_1} +\wp(b))=-2n_i+1$ for $i=1, \ldots, l$. Therefore, one is able to do the above simplification for $q>n_i$ over the same fraction field and obtains $f_{\Phi_1}$ in the below form
\begin{equation}
\label{eqnequicharZ/2form}
    \begin{split}
        f_{\Phi_1}&=\frac{X}{\prod_{i=1}^l (X-w_i)^{2n_i}} + \bigg( \sum_{i=1}^l \sum_{j=1}^{2n_i} \frac{c_{i,j}}{(X-w_i)^j} \bigg) ,\\
        &= \frac{\sum_{j=0}^{2q_1+1} a_j X^j}{\prod_{i=1}^l (X-w_i)^{2n_i}} =: \frac{N(X)}{\prod_{i=1}^l (X-w_i)^{2n_i}},
    \end{split}
\end{equation}
where $c_{i,j} \in \mathfrak{m}_A$, $a_j \in \mathfrak{m}_A$ for $j \neq 1$, and $a_1 \equiv 1 \pmod{\mathfrak{m}_A}$. Note that $N(X) \equiv X \pmod{\mathfrak{m}_A}$ and $N'(X) \equiv 1 \pmod{\mathfrak{m}_A}$. There hence exist $a \in \mathfrak{m}_A$ such that $N(a)=0$ due to Hensel's Lemma. One thus may further assume $a_0=0$ by replacing $X$ with $X-a$. The numerator of (\ref{eqnequicharZ/2form}) simplifies to $X (\sum_{j=1}^{2q_1+1}a_jX^{j-1})=:X \cdot I$. Finally, as $I$ is a unit in $R[X]$, the change of variables $X \mapsto X \cdot \sqrt[2q_1+1]{I}$ likes one in the case of mixed characteristic gives $f_{\Phi_1}$ the form 
\begin{equation*}
    f_{\Phi_1}= \frac{X}{\prod_{i=1}^l(X-(w_i+a)/\sqrt[2q_1+1]{I})^{2n_i}}=: \frac{X}{(X^{q_1+1}+v_1X^{q_1}+\ldots+v_{q_1+1})^2},
\end{equation*}
where the $v_i$ indeed have positive valuations. The identification $u_i \mapsto v_i$ then yields the specialization we seek.
\end{proof}

\begin{remark}
\label{remarkOSScomponent}
    The spectrum of $R_{m_1}$'s analog in odd characteristic $p$, where $m_1=pq-l$ $(1 \le l <p)$, also forms a full-dimensional component, a.k.a. the Oort-Sekiguchi-Suwa (OSS) component, of an Artin-Schreier extension's versal deformation ring as shown in \cite[Th{\'e}or{\`e}m 5.3.3]{MR1767273}. However, unlike characteristic $2$, there are deformations outside this component like most of those in \cite{DANG2020398} and \cite{2020arXiv200203719D}. More precisely, an equal-characteristic deformation lies in the OSS component only if the conductors (Definition \ref{defnconductor}) at all but may be one generic branch points are multiples of $p$ \cite[\S 3.1.1]{DANG2020398}. A mixed characteristic lift does so only if its ``Hurwitz data'' \cite[\S 2]{MR1630112}  is a length-$(m_1+1)$-multi-set of the form $\{1, \ldots, 1, l\}$.
\end{remark}

\section{The lift}
\label{secthelift}

\subsection{A minimal lift}
Recall that $\phi$ is given by a length-two-Witt-vector in (\ref{eqnphi}). That is to say, the corresponding extension of the function field $\kappa=k(x)$ is given by adjoining $y_1, y_2$ such that
\begin{equation}
\label{eqnexplicitASphi}
    \begin{cases}
        y_1^2-y_1 & =\frac{1}{x^{m_1}}, \text{ and} \\
        y_2^2-y_2 & = \frac{y_1}{x^{m_1}}+ \frac{\sum_{0 \le i < n_2} a_{n_2-i} x^i}{x^{n_2}}=:\frac{y_1}{x^{m_1}}+ f_2(x),
    \end{cases}
\end{equation}
where $n_2=2q_2+1$ is odd, $i$ is even, $a_{n_2-i} \in k$, and either $a_{n_2} \neq 0$ or all $a_{n_2-i}=0$, hence $f_2(x)=0$. The term $y_1/x^{m_1}$ comes from $S_1(X_1,Y_1, X_2, Y_2)-X_2-Y_2$ in \cite[\S 3]{MR2577662}, where $X_1=1/x^{m_1}$ and $Y_1=y_1$. Recall also that $\Phi_1$ is a lift a $\phi_1$ to a finite extension $R$ of $W(k)$ with uniformizer $\pi$. The equation represents $\Phi_1$ then can be expressed as in (\ref{eqnZ/2liftreducedform}) due to Proposition \ref{proprepresentationoforder2}. To simplify the calculation, we assume that one of $\Phi_1$'s generic branch points is $0$. That implies the polynomial $(\prod_{i=1}^{q_1+1}(X-v_i)^2+4 X \cdot G)$'s constant coefficient is zero, which in turn forces at least one of the $v_i$'s to be zero. We hence may rewrite (\ref{eqnZ/2liftreducedform}) as
\begin{equation}
\label{eqnreducedZ/2liftonebranch0}
    W_1^2=1+4 \frac{G}{X \prod_{i=1}^{q_1}(X-v_i)^2}.
\end{equation}
In addition, as in the proof of Proposition \ref{propASchar2versaldeformation}, substituting $W_1=1-2Y_1$ to (\ref{eqnreducedZ/2liftonebranch0}) gives an Artin-Schreier like equation
\begin{equation}
    \label{eqnZ/2liftAStype}
    (Y_1)^2-Y_1=\frac{G}{X\prod_{i=1}^{q_1}(X-v_i)^2}.
\end{equation}
Finally, let $R':=R_{\Phi_1}[\iu]$ and $\mathbb{K}':= \Frac (R' [X])$, where $R_{\Phi_1}$ is given in Proposition \ref{proprepresentationoforder2} and $\iu$ is a fixed primitive $4$-th root of unity. 
\begin{definition}
    \label{defniceZ2}
    We say a $\mathbb{Z}/2$ lift of $\phi_1$ is \emph{nice} if it can be represented by
    \begin{equation}
    \label{eqnniceZ2lift}
        W_1^2=1+4\frac{G}{X^{m_1}},
    \end{equation}
    where $G \in R[X]$.
\end{definition}

\begin{proposition}
   If $\Phi_1$ is nice, then all of its non-zero generic branch points have $2$-valuations of $2/m_1$.
\end{proposition}

\begin{proof}
As before, the non-zero generic branch points of $\Phi_1$ are the roots of $X^{m_1} + 4G$. The good reduction of $\Phi_1$ implies that $G \equiv 1 \pmod{2}$, and the degree of $G$ as a polynomial in $R[X]$ is strictly less than $m_1$. The proposition follows from the Newton polygon technique.
\end{proof}

\begin{remark}
\label{remarkconnectionOW}
    The above proposition suggests that the generic branch locus of $\Phi_1$ lies inside the open disc $D(1/m_1) \subset \mathbb{P}^{1,{\rm an}}_K$ containing points of valuation greater than $1/m_1$. This case is covered by \cite[Theorem 3.4 (i)]{MR3194815}.
\end{remark}

Theorem \ref{theoremlifttowerz4} is equivalent to the following.

\begin{theorem}
\label{theoremZ/4liftintowersexplicit}
    Given a $\mathbb{Z}/4$-extension $\phi$ as in (\ref{eqnexplicitASphi}) and a nice lift $\Phi_1$ of its $\mathbb{Z}/2$-subcover, which may be assumed to have the form (\ref{eqnniceZ2lift}). One may find a non-square rational $G_2 \in \mathbb{K}'$ such that the following equation
    \begin{equation}
        \label{eqnextendingPhi1}
        W_2^2=W_1 \cdot G_2= \bigg(1+4\frac{G}{X^{m_1}}\bigg) G_2
    \end{equation}
    defines a lift $\Phi$ of $\phi$ to $R'$ that extends $\Phi_1$. 
\end{theorem}

\begin{remark}
    The fact that $\Phi$ is generically cyclic of degree $4$ and extends $\Phi_1$ is immediate from Kummer theory. Additionally, the nice assumption makes the number of branch point
\end{remark}

 We will consider different cases of $f_2$'s degree separately.

\begin{proposition}
    \label{propminZ4lift}
The following system
\begin{equation}
\label{eqnminbirationalZ4lift}
    \begin{cases}
        W_1^2 & =1+4 \frac{G}{X \prod_{i=1}^{q_1}(X-v_i)^2}=:G_1 \\
        W_2^2 & =W_1  \bigg(1-2\iu \frac{G}{X\prod_{i=1}^{q_1}(X-v_i)^2} \bigg):=W_1 \cdot G_{2,\min} \\
    \end{cases}
\end{equation}
defines an extension $\Phi_{\min}$ of $\Phi_1$ that birationally lifts $\phi_{\min}:=\mathfrak{K}_2\big(\big( \frac{1}{x^{m_1}},0 \big)\big) \in {\rm H}^1_4(k(x))$. In particular, Theorem \ref{theoremZ/4liftintowersexplicit} is resolved when $f_2(x)$ in (\ref{eqnexplicitASphi}) equals $0$.
\end{proposition}
\begin{proof}
We apply the lifting technique of Green and Matignon \cite{MR1645000} as follows. Let $Y_2$ be the solution to  
\begin{equation}
    \label{eqndefineY2}
    W_2=1+(\iu-1)Y_1 -2 Y_2
\end{equation}
in $\mathbb{K}' (W_1,W_2)$. The right-hand-side of the above equation is precisely $F_1(Y_1)+\lambda Y_2$ in \cite[Proof of Theorem 6.16]{MR3051249} when the residue field's characteristic is $2$. Exercising the substitution (\ref{eqndefineY2}) and $W_1=1-2Y_1$ to (\ref{eqnextendingPhi1}) with $G_2=G_{2,\min}$ gives the following
\begin{equation*}
    \begin{split}
        1-4Y_2+4(Y_2)^2-2(1-\iu )Y_1&+4(1-\iu )Y_1Y_2-2\iu (Y_1)^2 \\&=1+4Y_1\iu\frac{G}{X\prod_{i=1}^{q_1}(X-v_i)^2} -2Y_1-2\iu\frac{G}{X\prod_{i=1}^{q_1}(X-v_i)^2}. 
    \end{split}
\end{equation*}
The equation, thanks to relation (\ref{eqnZ/2liftAStype}), simplifies to
\begin{equation}
\label{eqnminbirationalZ4liftAS}
    Y_2^2-Y_2+(1-\iu)Y_1Y_2=Y_1\iu\frac{G}{X\prod_{i=1}^{q_1}(X-v_i)^2}.
\end{equation}
One sees that (\ref{eqnminbirationalZ4liftAS}) reduces modulo $\pi$ to $y_2^2-y_2=y_1/x^{m_1}$. Thus, the special fiber of the cover defined by (\ref{eqnminbirationalZ4lift}) is birational to $\phi$.

Finally, after converting $G_{2,\min}$ to a fraction, the numerator $H_{\min}:= X\prod_{i=1}^{q_1}(X-v_i)^2-2iG \in R'[X]$ is of degree $m_1$, hence has at most $m_1$ distinct zeros, all specializing to $0$. Recall that, due to Kummer theory, all $\Phi_{\min}$'s branch points of branching index $2$ come from 
\begin{itemize}
    \item The zeros and poles of $G_{2,\min}$ of odd multiplicity and not that of $G_1$, and
    \item The zeros and poles of $G_1$ of multiplicity $2s$, where $s$ is odd.
\end{itemize}
One shows that the zeros and poles of $G_1$ of multiplicity $2s$ lie in a subset of $\{v_i\}_{i=1}^{q_1}$. In addition, it follows from Theorem \ref{theoremcaljumpirred} that $\phi_{\min}$ (resp., $\phi_1$) has jumps $(m_1, 2m_1)$ (resp., $m_1$) at its only branch point $0$. Therefore, Proposition \ref{propvanishingcycle}  asserts that the size of $\Phi_{\min}$'s (resp., $\Phi_1$'s) branch locus is at least $2m_1+1$ (resp., equal to $m_1+1$). If $\Phi_1$ is nice, i.e., $v_i=0$ for $i=1, \ldots, q_1$, then $\Phi_{\min}$'s branch points of index $2$ are precisely the zeros of $H_{\min}$ of odd multiplicity. Hence, $H_{\min}$ is separable, and $\Phi_{\min}$ has good reduction.
\end{proof}

\begin{corollary}
    \label{cormodifying}
    Suppose $f_2(x)$ in (\ref{eqnexplicitASphi}) is non-zero. Suppose, moreover, that there exists $G_2$ in $\mathbb{K}'$ such that the following conditions hold.
\begin{enumerate}[label=(\arabic*)]
        \item \label{cormodifying1} The number of $G_2$'s zeros and poles of odd multiplicity and distinct from those of $G_1$ is at most $m_2-m_1$.
        \item \label{cormodifying2} The character $\Phi' \in {\rm H}^1_2(\mathbb{K}')$ given by $G_2 \cdot G_{2, \min}$ has degeneration type 
        \begin{equation}
        \label{eqnmodifyingdegentype}
            \bigg(0, \frac{\sum_{0 \le i < n_2} a_{n_2-i} x^i}{x^{n_2}} \bigg).
        \end{equation}
    \end{enumerate}
    Then $G_2$ verifies Theorem \ref{theoremZ/4liftintowersexplicit}.
\end{corollary}

\begin{proof}
    Replacing $G_{2,\min}$ by $G_2$ equates to multiplying $\Phi_{\min}$ with the character generated by $G_2/G_{2,\min}$, which is in the same Kummer class of $G_{2} \cdot G_{2,\min}$,  hence isomorphic to $\Phi'$. Therefore, if $\Phi'$ has degeneration type (\ref{eqnmodifyingdegentype}), Lemma \ref{lemmacombination} \ref{lemmacombination4} indicates that $\Phi:=\Phi_{\min} \cdot \Phi'$'s special fiber is birational of $\phi$. In addition, due to Proposition \ref{propvanishingcycle}, the number of branch points of $\Phi$ with index $2$, denoted by $m$, is at least $m_2-m_1$. Recall that when $\Phi_1$ is nice, then the branch locus of $\Phi$ with index $2$ consists of the zeros and poles of $G_2$ with odd multiplicity and not those of $G_1$. Therefore, $m$ is equal to $m_2-m_1$ due to \ref{cormodifying1}.
\end{proof}

From this point on, we let $a_{j}$ denote the Teichm{\"u}ller lift of $f_2$'s coefficients $a_j$ to $W(k) \subset R'$, and $\sqrt{a_j} \in W(k)$ stand for that of $\sqrt{a_j} \in k$ (as $k$ is algebraically closed).

\subsection{When \texorpdfstring{$n_2 \le  m_1$}{}}
\label{secn2atmostm1}
In this case, $q_2 \le q_1$ and $m_2=2m_1$. The number of points with branching index $2$ is thus $m_1$.

\begin{proposition}
\label{propcaseq2smallerq1}
When $n_2 \le m_1$, Theorem \ref{theoremZ/4liftintowersexplicit} holds with $G_2$ is of the form
\begin{equation}
    \label{eqnZ4liftn2smallq2<q1}
    G_2=1- 2\iu\frac{G}{X\prod_{i=1}^{q_1}(X-v_i)^2} +4\frac{\sum_{0 \le i <n_2} a_{n_2-i} X^i}{X\prod_{j=1}^{q_2}(X-v_j)^2} \in \mathbb{K}' .
\end{equation}  
\end{proposition}

\begin{proof}
One shows that
\begin{equation*}
\begin{split}
     G_2 \cdot G_{2, \min}=&\bigg(1- 2\iu\frac{G}{X\prod_{i=1}^{q_1}(X-v_i)^2}\bigg)^2+4\frac{\sum_{0 \le i <n_2} a_{n_2-i} X^i}{X\prod_{j=1}^{q_2}(X-v_j)^2} \\ &-8\iu\frac{G(\sum_{0 \le j <n_2} a_{n_2-j} X^j)}{X^2\prod_{l=1}^{q_2}(X-v_l)^4 \prod_{m=q_2+1}^{q_1}(X-v_m)^2}.
\end{split}
\end{equation*}
Therefore, $\Phi':=\mathfrak{K}_2(G_2 \cdot G_{2, \min})$ has degeneration type (\ref{eqnmodifyingdegentype}) due to Theorem \ref{theoremswanorderpmixedchar}, where the correcting fuction is $H=1- 2i\frac{G}{X\prod_{i=1}^{q_1}(X-v_i)^2}$. In addition, the numerator of $G_2$ has degree $2q_1+1=m_1$, which implies that $\Phi$ has good reduction isomorophic to $\phi$ by Corollary \ref{cormodifying}.
\end{proof}

\subsection{When \texorpdfstring{$n_2>m_1$}{}}
\label{secn2greaterm1}
We have two sub-cases. When $m_1<n_2<2m_1$, we also have $m_2=2m_1$, hence, the number of points with branching index $2$ is still $m_1$. When $n_2>2m_1$, that amount is $2(q_2-q_1)$.

As before, we would like to forge a rational $G_2$ whose zeros and poles that contribute to the branch locus lie in the numerator, which should be a polynomial of degree $m_1$. One will see that it requires more effort to come up with the function compared to the case $q_2 \le q_1$. We will discuss a non-archimedean geometry interpretation of the construction in a forthcoming manuscript.

\begin{proposition}
\label{propcaseq2greatq1}
When $q_2>q_1$, Theorem \ref{theoremZ/4liftintowersexplicit} holds for the following
\begin{equation}
    \label{eqnZ4liftn2>n_1}
    \begin{split}
    G_2= &G_{2,\min}+2\frac{\sum_{0 \le i <n_2 -m_1} a_{n_2-i} X^i}{X^{\alpha}\prod_{j=1}^{\beta}(X-v_j)^2 }   +4\sum_{n_2-m_1 \le i <n_2} \frac{a_{n_2-i}}{X\prod_{j=1}^{q_2-i/2}(X-v_j)^2}  \\
    & + 2\sqrt{2} \frac{\sum_{0 \le i <n_2 -m_1} \sqrt{a_{n_2-i}} X^{i/2}}{X^{\alpha/2}\prod_{j=1}^{\beta}(X-v_j)} +4  \frac{\sum_{0 \le i<l <n_2 -m_1} \sqrt{a_{n_2-i}}\sqrt{a_{n_2-l}} X^{(i+l)/2}}{X^{\alpha}\prod_{m=1}^{\beta}(X-v_m)^2 }.
    \end{split}
\end{equation}  
where $\alpha  =\max \{0, 2(q_2-2q_1)\}$, and $\beta=q_2-q_1-\alpha/2$, hence is either $q_2-q_1$ or $q_1$.
\end{proposition}

\begin{proof}
A straight-forward computation shows that
\begin{equation}
\label{eqnq2>q1}
\begin{split}
    G_2&\cdot G_{2,\min} - \bigg(1-2\iu\frac{G}{X\prod_{i=1}^{q_1}(X-v_i)^2} +\sqrt{2}\iu \frac{\sum_{0 \le i <n_2 -m_1} \sqrt{a_{n_2-i}} X^{i/2}}{X^{\alpha/2}\prod_{j=1}^{\beta}(X-v_j) }  \bigg)^2 \\
    = & -4\iu\frac{G(\sum_{0 \le i <n_2 -m_1} a_{n_2-i} X^i)}{X^{\alpha+1}\prod_{l=1}^{\beta}(X-v_l)^2\prod_{j=1}^{q_1}(X-v_j)^2}  + 4\sum_{n_2-m_1 \le i <n_2} \frac{a_{n_2-i}}{X\prod_{j=1}^{q_2-i/2}(X-v_j)^2}  \\ & + \text{terms with $\nu_{\mathbb{K}'}$-valuations strictly greater than $2$.}     \\
\end{split}
\end{equation}
It hence follows from Theorem \ref{theoremswanorderpmixedchar} that the admissible character $\Phi'$ associated to $G_2 \cdot G_{2, \min}$ has degeneration type (\ref{eqnmodifyingdegentype}). In addition, the common denominator of (\ref{eqnZ4liftn2>n_1}) is 
\[ \begin{cases} 
      X^{\alpha}\prod_{i=1}^{\beta}(X-v_i)^2=X^{2(q_2-2q_1)} \prod_{i=1}^{q_1}(X-v_i)^2,  & q_2>2q_1 \\
      X \prod_{i=1}^{q_1}(X-v_i)^2,  & q_1<q_2 \le 2q_1   \\
   \end{cases}.
\]
Therefore, after converting $G_2$ to a fraction, the numerator is a polynomial in $R'[X]$ of degree 
\[ \begin{cases} 
      2(q_2-q_1),  & q_2>2q_1 \\
      2q_1+1,  & q_1<q_2 \le 2q_1   \\
   \end{cases}.
\]
The rest is instantaneous from Corollary \ref{cormodifying}.
\end{proof}

That completes the proof of Theorem \ref{theoremZ/4liftintowersexplicit}, hence Theorem \ref{theoremlifttowerz4}. \qed

\begin{remark}
    A straight forward application of Theorem \ref{theoremswanorderpmixedchar} yields that both of the $G_2$ in Proposition \ref{propcaseq2smallerq1}, Proposition \ref{propcaseq2greatq1}, and $G_{2,\min}$, all regarded as order-two-characters, have depth Swan $2$ and differential Swan $dx/x^{m_1+1}$. One can also deduce from Lemma \ref{lemmacombination} \ref{lemmacombination3} that the $\mathbb{Z}/4$-character defined by (\ref{eqnreducedZ/2liftonebranch0}) and $W_2^2=W_1$ has the exact same degeneration type.
\end{remark}

\bibliographystyle{alpha}
\bibliography{mybib}

\end{document}